\documentclass{patmorin}
\usepackage{pat}
\usepackage[T1]{fontenc}
\usepackage[utf8]{inputenc}
\usepackage{todonotes,float,subcaption}
\usepackage{algorithm,algpseudocode}
\usepackage{comment}
\usepackage{indentfirst}
\DeclareMathOperator{\FAlt}{FAlt}
\DeclareMathOperator{\FSym}{FSym}
\DeclareMathOperator{\Sym}{Sym}
\DeclareMathOperator{\Alt}{Alt}
\DeclareMathOperator{\PSL}{PSL}
\DeclareMathOperator{\GL}{GL}

\DeclareMathOperator{\gn}{gn}
\DeclareMathOperator{\ugn}{ugn}
\DeclareMathOperator{\stww}{stww}
\DeclareMathOperator{\tww}{tww}
\DeclareMathOperator{\utww}{utww}
\DeclareMathOperator{\dw}{dw}

\newcommand{\C}{\mathcal C}

\usepackage[mathlines]{lineno}
\newcommand*\patchAmsMathEnvironmentForLineno[1]{%
 \expandafter\let\csname old#1\expandafter\endcsname\csname #1\endcsname
 \expandafter\let\csname oldend#1\expandafter\endcsname\csname end#1\endcsname
 \renewenvironment{#1}%
    {\linenomath\csname old#1\endcsname}%
    {\csname oldend#1\endcsname\endlinenomath}}%
\newcommand*\patchBothAmsMathEnvironmentsForLineno[1]{%
 \patchAmsMathEnvironmentForLineno{#1}%
 \patchAmsMathEnvironmentForLineno{#1*}}%
\AtBeginDocument{%
\patchBothAmsMathEnvironmentsForLineno{equation}%
\patchBothAmsMathEnvironmentsForLineno{align}%
\patchBothAmsMathEnvironmentsForLineno{flalign}%
\patchBothAmsMathEnvironmentsForLineno{alignat}%
\patchBothAmsMathEnvironmentsForLineno{gather}%
\patchBothAmsMathEnvironmentsForLineno{multline}%
}

\usepackage[longnamesfirst,numbers,sort&compress]{natbib}

\definecolor{brightmaroon}{rgb}{0.76, 0.13, 0.28}
\definecolor{linkblue}{rgb}{0, 0.337, 0.227}
\newcommand{\defin}[1]{\emph{\textcolor{brightmaroon}{#1}}}
\makeatletter
\def\mathcolor#1#{\@mathcolor{#1}}
\def\@mathcolor#1#2#3{%
  \protect\leavevmode
  \begingroup
    \color#1{#2}#3%
  \endgroup
}
\makeatother

\usepackage[inline]{enumitem}
\newcommand{\sm}{\smallsetminus}

\title{Quasirandomness and Uniform Twin-Width}
\author{
George Kontogeorgiou\thanks{Center for Mathematical Modeling (CNRS IRL2807), University of Chile, Santiago, Chile. \\Supported by ANID Basal Grant CMM FB210005 and ANID-FONDECYT Postdoctorado Grant No. 3250479. \\Email:{ \tt gkontogeorgiou@dim.uchile.cl}
}\hspace{5mm}Bobby Miraftab\thanks{School of Computer Science, Carleton University, Ottawa, Canada. Email: {\tt bobby.miraftab@gmail.com}}}
\begin{document}

\maketitle
\begin{abstract}
    For every nontrivial finite group, we prove that its quasirandom degree gives a polynomial lower bound on its uniform twin-width, whereas its minimum faithful complex representation degree gives a linear upper bound. For nonabelian finite simple groups, these two parameters coincide, so uniform twin-width is polynomially equivalent to quasi-randomness in that class, yielding a new definition of quasirandomness in the sense of Gowers. We use the lower bound to prove that uniform twin-width is unbounded over finite groups, which helps us construct finitely presented groups with finite twin-width but infinite uniform twin-width. This answers a question of Bonnet, Geniet, Tessera and Thomass\'e. Finally, we determine the uniform twin-width of all three Thompson groups.
\end{abstract}
%%%%%%%%%%%%%%%%%%%%%%%%%%%%%%%%%%%%%%%%%%%%%%%%%%%%%%%%%%%%%%%%%%%%%%%%%%%%%%%%%%%%%%%%%%%%%%%%%%%%
\section{Introduction}

\defin{Twin-width} is a graph parameter first introduced in \cite{bonnet2021twin}. It is a measure of the complexity of a graph defined through contraction sequences. The study of twin-width has revealed deep connections with many other concepts within structural and algorithmic graph theory, including FO-model checking, graph colorings, and the area of algorithmic meta-theorems.

Twin-width is also a meaningful concept in geometric group theory; the twin-width of an infinite graph (and in particular of a Cayley graph) is defined as the supremum of the twin-widths of all its finite induced subgraphs. In Twin-width II \cite{bonnet2021twin2}, Bonnet, Geniet, Kim and Thomass\'e proved that finiteness of twin-width is a group invariant for finitely generated groups. That is, either every (locally finite) Cayley graph of a group has finite twin-width, or none. One may therefore say that a finitely generated group has finite or infinite twin-width, respectively. The question of whether there exist finitely generated groups of infinite twin-width was also posed in \cite{bonnet2021twin2} and was later answered in \cite{BGTT} in the affirmative. We remark that a useful characterization of the finitely generated groups of infinite twin-width remains unknown. In \cite{kontogeorgiou2026accessibility}, the authors showed that finiteness of twin-width is unrelated to accessibility.   

Twin-width VII \cite{BGTT} introduced a new approach to twin-width for groups. By coming up with new notions of twin-width for permutation matrices, for group actions, and eventually for groups, Bonnet, Geniet, Tessera and Thomass\'e arrived at the definition of \emph{uniform twin-width}. Unlike twin-width, uniform twin-width is a well-defined parameter for groups that takes values in $\mathbb{N}\cup\{\infty\}$. Moreover, finiteness of uniform twin-width is preserved under many standard group operations, including extensions, products, direct sums, direct limits, finite-index supergroups, and quotients by finite normal subgroups. Finally, if a group has finite uniform twin-width, then it also has finite twin-width. 

The authors of \cite{BGTT} conjectured that the converse of the latter statement is false. They proposed as a candidate counterexample the \emph{lampshuffler group}, which has finite twin-width, and offered some preliminary evidence that it has infinite uniform twin-width:

\begin{thm}{\rm\cite[Theorem 6.19]{BGTT}}\label{known:unbounded}
    The lampshuffler group $L$ has finite uniform twin-width if and only if there is a uniform bound on the uniform twin-widths of all finite groups. 
\end{thm}

In this paper, we verify the intuition of Bonnet, Geniet, Tessera and Thomass\'e. In fact, we prove something more general. 

\begin{thm}\label{thm:Houghton}

Every Houghton group $H_n$ with $n\geq2$ has finite twin-width but infinite uniform twin-width.

\end{thm}

In particular, this proves the existence of finitely presented groups with this property \cite{brown1987finiteness}. We note that these examples are non-simple and elementary amenable. For this reason, we also provide an example that is simple and non-amenable.

\begin{thm}
    Thompson's group $V$ has finite twin-width but infinite uniform twin-width.
\end{thm}

We note that, even for finitely presented groups, we do not know of any interesting group-theoretic conditions under which finiteness of twin-width implies finiteness of uniform twin-width. In fact, as seen above, we rule out two such conditions. We leave this as an open problem. 

\paragraph{Methods and outline.}The main observation of this paper is that, for finite simple groups, \textbf{uniform twin-width is an equivalent definition of quasirandomness.} Indeed, we bound the uniform twin-width of any finite simple group $G$ by polynomial functions of the minimum dimension of a nontrivial irreducible complex representation of $G$. The latter is one of the original measures of quasirandomness introduced by Gowers \cite{Gowers}. We also prove some more general bounds that are true for all finite groups. Explicitly, our main theorem is as follows. For a nontrivial finite group $G$, let $D_{\mathrm{qr}}(G)$ denote the minimum degree of a nontrivial irreducible complex representation of $G$, and let $D_{\mathrm f}(G)$ denote the minimum degree of a faithful complex representation of $G$.

\begin{thm}\label{thm:quasirandomness}
There exist absolute constants $c,C>0$ such that every nontrivial finite group $G$ of order $N$ satisfies
\[
  cD_{\mathrm{qr}}(G)^{1/4}
  \leq \utww(G)
  \leq CD_{\mathrm f}(G)\qquad
\text{and}\qquad \utww(G)\leq C N^{3/14}.\]
If $G$ is nonabelian simple, then $D_{\mathrm{f}}(G)=D_{\mathrm qr}(G)\geq \frac{\sqrt{\log N}}{2}$, and consequently
\[
  cD_{\mathrm{qr}}(G)^{1/4}
  \leq \utww(G)
  \leq CD_{\mathrm{qr}}(G)\qquad
\textit{and}\qquad
  c(\log N)^{1/8}
  \leq \utww(G)
  \leq C N^{3/14}.
\]
\end{thm}

We hope that this discovery may offer additional motivation for studying twin-width in the context of groups.

In \cref{sec3}, starting from Gowers' mixing lemma for quasirandom groups (\cref{known:gowers-mixing}), we prove \cref{thm:quasirandomness}, which in turn implies that the uniform twin-width of finite groups is unbounded. In \cref{sec4} we apply this result to prove \cref{thm:Houghton}. Finally, in pursuit of more examples, in \cref{sec5} we compute the uniform twin-width of all three Thompson's groups. We now proceed to \cref{sec2}, where we collect many prerequisites for our work.

\section{Preliminaries}\label{sec2}

Throughout this paper, we use the notation $[k]:=\{1,\dots,k\}$.

\paragraph{Twin-width of graphs and groups.} A \defin{vertex identification} in a graph $G$ is an operation that deletes two distinct vertices $u,v\in V(G)$ and introduces a new vertex $w$ with
\[
  N(w)=\bigl(N(u)\cup N(v)\bigr)\sm\{u,v\}.
\]
For a finite graph $G$ on $n$ vertices, consider a sequence
\[
  G=G_n\longrightarrow G_{n-1}\longrightarrow\cdots\longrightarrow G_1
\]
where each $G_i$ is obtained from $G_{i+1}$ by a vertex identification. Its \defin{width} is $\max_{1\leq i\leq n}\Delta(G_i)$. The minimum width among all vertex identification sequences is the \defin{strict twin-width} of $G$, denoted $\stww(G)$. The strict twin-width of an infinite graph is defined to be the supremum of the strict twin-widths of its finite induced subgraphs.

Strict twin-width was introduced in \cite{BGTT}. It is simpler to define than ordinary twin-width, denoted $\tww(G)$, while remaining equally suitable for studying Cayley graphs. Specifically,

\begin{lem}
{\rm \cite[Equation (1), Section 3.1.1]{BGTT}}
\label{lem:strict-ordinary}
for every finite group $G$,
\[
        \max\{\Delta(G),\operatorname{tww}(G)\}
        \leq
        \operatorname{\stww}(G)
        \leq
        \Delta(G)+\operatorname{tww}(G),
\]
\end{lem} 
\hspace{-7mm}which implies that, for graphs of bounded degree, finiteness of twin-width is equivalent to finiteness of strict twin-width. We therefore opt to work only with strict twin-width in this paper, and shall always refer to it merely as "twin-width".

The following result is foundational.

\begin{lem}{\rm\cite[Lemma 42]{bonnet2021twin2}}
    Either all Cayley graphs of a finitely generated group have finite twin-width, or none.
\end{lem}

Accordingly, one may say that a finitely generated group has finite or infinite twin-width, although ordinary twin-width is not a numerical parameter of the group itself. To define such a parameter, we first turn to ordered matrices. Twin-width for matrices historically predates twin-width for graphs, having been introduced by Guillemot and Marx in \cite{guillemot2014finding}.

\paragraph{Twin-width and grid number of binary matrices.} 

A matrix (of any cardinality) is \defin{binary} if each of its entries is $0$ or $1$, and it is a \defin{bijection matrix} if each row and each column contains exactly one entry equal to $1$. A \defin{row/column identification} on a matrix is an operation in which two consecutive rows or columns $v_1$ and $v_2$ are superimposed. That is, they are replaced by a single row or column (respectively), each entry of which is equal to the maximum of the corresponding entries in $v_1$ and $v_2$. Given a finite $m\times n$ matrix $M$ and a row/column identification sequence $M:=M_{m+n},\dots,M_2$, we can define the latter's width as the maximum number of $1$'s in a single row or column in any matrix of the sequence. The minimum width among all row/column identification sequences for a matrix $M$ is defined to be the strict twin-width (henceforth merely called "twin-width") of $M$. The twin-width of an infinite matrix is defined to be the supremum of the twin-widths of all its finite submatrices.      

Let $(X,<_{X})$ and $(Y,<_{Y})$ be totally ordered sets.  
A \defin{$k$-grid} in a binary $X\times Y$ matrix $M$ consists of order-preserving partitions $X=X_1\sqcup\cdots\sqcup X_k$, and $Y=Y_1\sqcup\cdots\sqcup Y_k$
into non-empty intervals such that every \defin{zone} $X_i\times Y_j$ contains a $1$.  
The \defin{grid number} $\gn(M)$ is the
largest $k$ for which $M$ has a $k$-grid, or infinity if no finite maximum exists. A very useful result is that twin-width and grid number are bounded by functions of each other. Recall that the notation $f(r)\sim g(r)$ means that
$\lim_{r\to\infty}\frac{f(r)}{g(r)}=1$.

\begin{lem}{\rm \cite[Theorem~3.2, with degree bound $d=1$]{BGTT}}\label{known:grid-tww}
There are non-decreasing functions $F,\Gamma\colon\N\to\N$ such that
$F(r)=2^{O(r)}$, and $\Gamma(r)\sim 2r$,
and every finite or infinite bijection matrix $M$ satisfies
\[
  \stww(M)\leq F(\gn(M)),
  \qquad
  \gn(M)\leq \Gamma(\stww(M)),
\]
whenever the quantities on the left are finite.
\end{lem}

We remark (see \cite[Remark 3.1]{BGTT}) that, for every $m\times m$ permutation matrix $M$, we have $\gn(M)\leq \sqrt{m}$, hence $\stww(M)\leq 2^{O(\sqrt{m})}$. In fact, let us prove a sharper bound for $\stww(M)$.

\begin{lem}\label{lem:gridbound}
Every $m\times m$ permutation matrix $P$, with arbitrary fixed orders on its rows and columns, satisfies $\operatorname{stww}(P)\leq \lceil\sqrt m\rceil$.
\end{lem}

\begin{proof}
Put $a=\lceil\sqrt m\rceil$. First we partition the rows into sets of size at most $a$ and we merge the rows in each part. Throughout this phase, every row contains at most $a$ nonzero elements, while every column contains exactly one nonzero element.
There are now
\[
  b=\left\lceil\frac{m}{a}\right\rceil\leq a
\]
row parts.  We then merge all columns into one.  During this phase, every row contains at most $a$ nonzero elements, and every column contains at most
$b\leq a$ elements.  Finally, we merge the $b$ row parts.  Every matrix in the resulting row/column identification sequence has at most $a$ nonzero elements in each row and column.
\end{proof}

Here is another useful lemma regarding the twin-width of matrices.

\begin{lem}{\rm \cite[Lemma~3.5]{BGTT}}\label{lem:sub}
Suppose that a bijection matrix $M'$ is obtained from a bijection
matrix $M$ by substituting a bijection matrix $N_{ij}$ for every
nonzero entry $(i,j)$ of $M$. Then
\[
  \operatorname{stww}(M')
  =
  \max\!\left\{
      \operatorname{stww}(M),
      \max_{M(i,j)=1}\operatorname{stww}(N_{ij})
  \right\}.
\]
This remains valid for infinite ordered bijection matrices.
\end{lem}

Finally, we introduce another easy-to-handle parameter that correlates well with grid number. Recall that for any bijection\footnote{We say "bijection" instead of "permutation" because the total orders on $X$ and $Y$ may not be isomorphic.}
$\sigma\colon X\to Y$ we can define a bijection matrix $M_{\sigma}$ with
\[
  M_\sigma(x,y)=1 \quad\Longleftrightarrow\quad \sigma(x)=y
\]
and with its rows and columns ordered isomorphically to $<_X$ and $<_Y$, respectively.
For a permutation $\sigma$ of a totally ordered set $(X,<)$, its \defin{decreasing width} is
\[
  \dw_{<}(\sigma)
  :=\sup\bigl\{r:\exists x_1<\cdots<x_r
     \text{ with }\sigma(x_1)>\cdots>\sigma(x_r)\bigr\}.
\]

\begin{lem}\label{lem:grid-decreasing}
For every permutation $\sigma$ of a totally ordered set, $\gn(M_\sigma)\leq \dw_{<}(\sigma)$.
\end{lem}

\begin{proof}
Let $X=X_1\sqcup\cdots\sqcup X_k$, $Y=Y_1\sqcup\cdots\sqcup Y_k$ be a $k$-grid of $M_{\sigma}$. For each $i\in [k]$, choose a $1$-entry in the zone $X_i\times Y_{k+1-i}$ and let $x_i\in X_i$ be its row. Then $x_1<\dots<x_k$ but $\sigma(x_1)>\dots>\sigma(x_k)$, because $\sigma(x_i)\in Y_{k+1-i}$ for each $i\in [k]$. Therefore $\dw_<(\sigma)\geq k$. 
\end{proof}

\begin{lem}\label{lem:monotone-cover}
Suppose that a permutation $\sigma$ of a totally ordered set can be partitioned into $r$ increasing sequences. Then $\dw_{<}(\sigma)\leq r$.
\end{lem}

\begin{proof}
A decreasing sequence intersects an increasing sequence at most once, so each element of a longest decreasing sequence lies in a different increasing sequence of the assumed partition, whence the statement follows.
\end{proof}

\paragraph{Uniform twin-width of group actions and groups.} Let $(X,<)$ be a totally ordered set and let $S_X$ be the group of permutations of $X$. Consider a group $G$ and an action $\phi\colon G\rightarrow S_X$. Each $g\in G$ gives rise to a permutation matrix $M_{\phi(g)}$. We say that $\phi$ has \emph{finite twin-width} if there exists a choice of total order for $X$ so that $\stww\left(M_{\phi(g)}\right)<\infty$ for every $g\in G$. Moreover, if $\gamma$ is the right canonical self-action of $G$, then we say that $G$ has \defin{finite twin-width} if $\gamma$ has finite twin-width. This definition expands the definition from the perspective of Cayley graphs that we have for finitely generated groups (see \cite[Lemma 5.1]{BGTT}).

The \emph{uniform twin-width} of the action $\phi$ \emph{with respect to the order} $<$ is given by 
\[\utww_<(\phi):=\sup_{g\in G}\stww(M_{\phi(g)}).\]
The \emph{uniform twin-width} of $\phi$ is defined as
\[\utww(\phi):=\min_{< \text{ is a total order on $X$ }}\utww_<(\phi).\]
Finally, we define the \defin{uniform twin-width} of $G$ as $\utww(G):=\utww(\gamma)$.
 
As mentioned in the introduction, a group with finite uniform twin-width also has finite twin-width. We aim to provide a counterexample to the converse. Towards this goal we mention the following properties of uniform twin-width.

\begin{lem}{\rm\cite[Proposition~6.4]{BGTT}}\label{lem:abelian}
    Every finite abelian group has uniform twin-width $2$.
\end{lem}

\begin{lem}{\rm\cite[Lemma~6.1]{BGTT}}\label{known:subgroup}
If $H\leq G$, then $\utww(H)\leq \utww(G)$.
\end{lem}

\begin{lem}{\rm \cite[Corollary~6.6]{BGTT}}\label{lem:extension}
If $N\triangleleft G$, then we have 
$\operatorname{utww}(G)
  \leq
  \max\!\left\{
      \operatorname{utww}(N),
      \operatorname{utww}(G/N)
  \right\}$.
\end{lem}

\begin{lem}{\rm\cite[Corollary~6.17]{BGTT}}\label{known:faithful-action}
Let $\phi$ be a faithful action of a group $G$ on a well-ordered set $(X,<)$. If for each $g\in G$, the matrix $M_{\phi(g)}$ has finite twin-width, then $G$ has finite twin-width. If for all $g\in G$ we have $\stww(M_{\phi(g)})\leq d$, then $\utww(G)\leq d$.
\end{lem}

\begin{thm}\label{thm:finite-shuffle}
Let $\phi$ be a faithful action of a group $G$ on a well-ordered set $(X,<)$.
\begin{enumerate}[label=(\alph*)]
\item If $\dw_{<}(\phi(g))<\infty$ for every $g\in G$, then $G$ has finite twin-width.
\item If $\sup_{g\in G}\dw_{<}(\phi(g))\leq r$, then
$\utww(G)\leq F(r)$,
where $F$ is the function in~\Cref{known:grid-tww}.
\end{enumerate}
\end{thm}
\begin{proof}
By~\Cref{lem:grid-decreasing},
$\gn(M_{\phi(g)})\leq\dw_{<}(\phi(g))$.  Then~\Cref{known:grid-tww} gives finite twin-width for each $M_{\phi(g)}$, and the uniform bound $F(r)$ in case (b).
We finish case (b) by invoking~\Cref{known:faithful-action}.
\end{proof}

For a nontrivial finite group $G$, we define its \defin{minimum faithful permutation degree} by \[\mu(G)\coloneqq \min\{n\in\mathbb{N}|\text{ }
  \phi:G\rightarrow \Sym(n)\text{ is a faithful permutation representation of }G\}.\]
Combining \cref{lem:gridbound} and \cref{known:faithful-action}, we get:

\begin{lem}\label{lem:permubound}
    For every nontrivial finite group $G$, $\utww(G)\leq \lceil\sqrt{\mu(G)}\rceil$.
\end{lem}

Finally, for a group $G$, we define the
\defin{uniform grid number}
\[
  \ugn(G):=\min_{<\text{ is a total order on }G}\ \sup_{g\in G}\gn(M_{\gamma(g)}).
\]

It follows easily from the definitions and from \cref{known:grid-tww} that $\utww(G)\leq F(\ugn(G))$ and $\ugn(G)\leq \Gamma(\utww(G))$.

\paragraph{Quasirandom groups.} Recall that, for a nontrivial finite group $G$, we define its \defin{quasirandom degree} and
\defin{minimum faithful complex degree}, respectively, by
\begin{align*}
D_{\mathrm{qr}}(G)
&:=\min\bigl\{\dim_{\mathbb C}V:
  V\text{ is a nontrivial irreducible complex }G\text{-module}\bigr\},\\
D_{\mathrm f}(G)
&:=\min\bigl\{\dim_{\mathbb C}V:
  G\hookrightarrow\GL(V)\bigr\}.
\end{align*}
We call a finite group \defin{$D$-quasirandom} if $D_{\mathrm{qr}}(G)\geq D$.

Quasirandom groups were introduced by Gowers \cite{Gowers} as a means to tackle a question of Babai and S\'os about whether every group contains a product-free subset of linear size. The answer is negative. The offered counterexamples are quasirandom, namely the groups $PSL_2(q)$. We isolate some particularly useful lemmata.

\begin{lem}{\rm \cite[Lemma~5.1 and the paragraph immediately following it]{Gowers}}\label{known:gowers-mixing}
Let $G$ be a finite group.  If $A,B,C\subseteq G$ contain
no solution $ab=c,a\in A, b\in B$ and $c\in C$,
then
\[
  |A|\,|B|\,|C|\leq \frac{|G|^3}{D_{qr}(G)}.
\]
\end{lem}

\begin{lem}{\rm \cite[Theorem 4.7]{Gowers}}\label{known:gowers-representation}
    Let $G$ be a finite, simple, nonabelian group. Then \[D_{qr}(G)\geq\frac{\sqrt{\log|G|}}{2}.\] 
\end{lem}
For more details, we direct the reader to the original paper by Gowers \cite{Gowers}.
%%%%%%%%%%%%%%%%%%%%%%%%%%%%%%%%%%%%%%%%%%%%%%%%%%%%%%%%%%%%%%%%%%%%%%%%%%%%%%%%%%%%%%%%%%%%%%%%%%%

\section{Finite groups}\label{sec3}

In this section we prove that finite groups have unbounded uniform twin-width and we quantify this result with respect to $N:=|G|$, $D_{qr}:=D_{qr}(G)$, $D_{f}:=D_f(G)$.

\begin{thm}\label{thm:transitions}
Let $G=P_1\sqcup\cdots\sqcup P_k$ be a partition satisfying
$|P_i|\geq \frac{N}{2k}$ for every $i\in[k]$.
Then we have
\[
 \left|
 \left\{g\in G:
 P_i g\cap P_j\neq\varnothing
 \text{ for every }i,j\in[k]
 \right\}
 \right|
 \geq N\left(1-\frac{4k^4}D_{\hspace{-2mm}qr}\right).
\]
\end{thm}

\begin{proof}
For $i,j\in[k]$, let $E_{ij}\coloneqq \{g\in G:P_i g\cap P_j=\varnothing\}
          =G\sm P_i^{-1}P_j$.
That is, there is no solution $ab=c$ with
$a\in P_i^{-1}$, $b\in P_j$, and $c\in E_{ij}$. 
It follows from~\Cref{known:gowers-mixing} that,
\[
  |E_{ij}|
  \leq \frac{N^3}{D_{qr}|P_i||P_j|}
  \leq \frac{4k^2N}{D_{qr}}.
\]
The union bound over the $k^2$ ordered pairs gives
\[
  \left|\bigcup_{i,j}E_{ij}\right|
  \leq \frac{4k^4N}{D_{qr}}.
\]
The complement is exactly the set in the statement.
\end{proof}

\begin{thm}\label{thm:D-utww}
Every nontrivial finite group $G$ satisfies
\[
  D_{\mathrm{qr}}(G)
  \leq4\bigl(\ugn(G)+1\bigr)^4
  \leq4\bigl(\Gamma(\utww(G))+1\bigr)^4.
\]
Consequently,
$\utww(G)=\Omega(D_{\mathrm{qr}}^{1/4})$.
In particular, every finite nonabelian simple group $G$ satisfies
\[
  \utww(G)=\Omega\bigl((\log N)^{1/8}\bigr).
\]
\end{thm}

\begin{proof}
Let $u=\ugn(G)$ and choose an order on $G$ that realizes $u$.  Put $k=u+1$.
Suppose that $D_{qr}>4k^4$ and consider an order-preserving partition of $G$ into $k$ intervals of almost equal size.  
\Cref{thm:transitions} then gives at least one translation
whose matrix has a $k$-grid, contradicting the definition of $u$. Thus we obtain the displayed inequality. 
As $\Gamma(r)\sim2r$, solving for $\utww(G)$ gives the claimed lower bound with respect to $D_{qr}$. The bound with respect to $N$ for simple, nonabelian groups comes from \cref{known:gowers-representation}.
\end{proof}

\begin{cor}\label{cor:finite-unbounded}
Uniform twin-width is unbounded over finite groups.
\end{cor}

\begin{proof}
The orders of the nonabelian simple groups $\Alt(n)$ tend to infinity. Hence, \cref{thm:D-utww} implies that $\utww(\Alt(n))\to\infty$.
\end{proof}

There are groups that have polynomial uniform twin-width with respect to their size. 
For example, $D_{qr}(\PSL_2(q))\geq \frac{q-1}{2}$ (see, e.g., \cite[Theorem 3.1]{Gowers}), so $\utww(\PSL_2(q))=\Omega(N^{1/12})$. The upper bound that we obtain for uniform twin-width with respect to $N$ is even more than that. We begin by showing that every nonabelian finite simple group has a faithful permutation representation of degree $O(|S|^{3/7})$.

\begin{lem}\label{lem:permurep}
For each nonabelian finite simple group $S$, we have
$\mu(S)=O(|S|^{3/7})$, hence $\utww(S)=O(|S|^{3/14})$.
\end{lem}

\begin{proof}

We distinguish cases regarding $S$ (as we only care about asymptotics, we focus on the infinite families). If $S=\Alt(r)$, the natural action of $S$ on $\{1,\dots,r\}$ gives $\mu(S)\leq r$, so $\mu(S)=o(\log |S|)=O(|S|^{3/7})$. For the finite simple groups of Lie type, the values of $\mu(S)$ are listed in \cite[Section~3, Table~4]{GuestMorrisPraegerSpiga}.
Comparison with the corresponding order formulas gives again
$\mu(S)=O(|S|^{3/7})$. The second bound follows from \cref{lem:permubound}. 
\end{proof}

As we will see, we can extend the second bound to all finite groups. In fact, we can also get an upper bound on uniform twin-width as a function of $D_f$. We will first require some additional lemmata.

\begin{lem}\label{lem:simple_subgroup}
    If $H$ is a subgroup of a finite, simple, nonabelian group $S$ such that $[S:H]$ is minimal, then $\mu(S)=[S:H]$.
\end{lem}

\begin{proof}
First observe that $\mu(S)=\min_{H<S}[S:H]$.
Indeed, for every proper subgroup $H<S$, the kernel of the coset action on $S/H$ is the core
\[
  \operatorname{core}_S(H)=\bigcap_{s\in S}sHs^{-1},
\]
which is a proper normal subgroup of the simple group $S$, and hence is trivial. Conversely, if $S$ acts faithfully on a finite set $\Omega$, then it has a nontrivial orbit $S\omega$. Its stabilizer $S_\omega$ is proper, and therefore
\[
  |\Omega|\geq|S\omega|=[S:S_\omega]
  \geq\min_{H<S}[S:H].\qedhere
\]
\end{proof}

\begin{lem}[{Nikolov--Pyber,\text{ }\cite{nikolov2011product}}]\label{known:NP-permutation}
    Let $G$ be a finite linear group of degree $k$ over $\C$. Then $G$ has a permutation representation of degree at most $c_0k^2$ with abelian kernel for some absolute constant $c_0$. 
\end{lem}

We are now ready to prove upper bounds for the uniform twin-width of finite groups. 

\begin{thm}\label{thm:upperbound}
For every finite group
$G$, we have $\operatorname{utww}(G)=O(D_f)$ and
$\operatorname{utww}(G)=O(N^{3/14})$.
\end{thm}

\begin{proof}
We first apply \cref{known:NP-permutation} to $G$ and we obtain a homomorphism $G\rightarrow \Sym(m)$ for some $m=O(D_{f}^2)$. Then $\phi:\utww(G/ker(\phi))\leq\sqrt{m}=O(D_f)$ by \cref{lem:gridbound}, and $\utww(\ker(\phi))=2$ by \cref{lem:abelian}. By \cref{lem:extension}, we obtain the first inequality.

Now let $1=G_{0}\triangleleft G_{1}\triangleleft\cdots
    \triangleleft G_{r}=G$
be a composition series, and put $S_i=G_i/G_{i-1}$.  
Iterating \cref{lem:extension} gives
\[
  \operatorname{utww}(G)
     \leq \max_{1\leq i\leq r}\operatorname{utww}(S_i).
\]

If $S_i$ is cyclic of prime order, then by \cref{lem:abelian} we have that
$\operatorname{utww}(S_i)=2$.  Otherwise, $S_i$ is nonabelian simple, so $\utww(S_i)=O(|S_i|^{3/14})$ by \cref{lem:permurep}. As $|S_i|\leq N$, we get the second inequality.
\end{proof}

It is important to note that, for arbitrary finite groups, the two representation parameters $D_{qr}$ and $D_f$ cannot be conflated, and in particular there does not exist an upper bound of $\utww(G)$ with respect to $D_{qr}$. For example,
$D_{\mathrm{qr}}(\Alt(n)\times C_2)=1$,
whereas
\[
  \utww(\Alt(n)\times C_2)
  \geq\utww(\Alt(n))\longrightarrow\infty
\]
by \cref{known:subgroup,cor:finite-unbounded}. The polynomial equivalence of uniform twin-width and quasirandom degree is restricted to finite simple groups.

For the various bounds that we produced in this section, it would be interesting to find whether they can be attained, especially for the class of nonabelian finite simple groups. Specifically, regarding the asymptotics with respect to $N$, it is easy to see for the alternating groups $\Alt(n)$ that $\utww(\Alt(n))=O(D_{qr})=O\left(\frac{\log N}{\log\log N}\right)$, which however is still far from the bound $\Omega\left((\log N)^{1/8}\right)$ of \cref{thm:D-utww}. As for the bound $O(N^{3/14})$ of \cref{thm:upperbound}, a result of Barbieri and Sabatini {\rm\cite[Corollary C]{barbieri2024quasirandom}} offers hope that we may even find witnesses for tightness among general quasisimple groups.   

\begin{comment}
\begin{itemize}
    \item $\utww(G)=\Omega\left((\log N)^{1/8}\right)$: unknown. However, it is certainly not very loose. The groups $\Alt(n)$ are not of prime order, so, as explained in the first item of this list, $\utww(\Alt(n))=\Theta(D^{1/4})$. 
    \item $\utww(G)=O(N^{3/14})$: the bound $\mu(S)=O(N^{3/7})$ of \cref{lem:permurep} is sharp, but the derived bound $\utww(G)=O(N^{3/14})$ may not be. However, 
\end{itemize}
\end{comment}

%%%%%%%%%%%%%%%%%%%%%%%%%%%%%%%%%%%%%%%%%%%%%%%%%%%%%%%%%%%%%%%%%%%%%%%%%%%%%%%%%%%%%%%%%%%%%%%%%%%%
\section{Houghton groups}\label{sec4}

For $n\geq2$, let $S_n$ be the infinite star with vertex set
$V(S_n)=\{0\}\cup\{(i,m):i\in[n],\ m\in\mathbb N_{\geq1}\}$,
where $0$ is adjacent to each $(i,1)$ and $(i,m)$ is adjacent to $(i,m+1)$. Let $d_{S_n}$ denote its graph metric and let $\Sym(S_n)$ be the full permutation group of $V(S_n)$. Define the \defin{wobbling group}
\[
  W(S_n):=
  \left\{\sigma\in\Sym(S_n):
  \|\sigma\|_w:=
  \sup_{v\in V(S_n)}d_{S_n}(\sigma(v),v)<\infty\right\}.
\]
For $n=2$, the graph $S_2$ is a double ray, which we identify with $\mathbb Z$; thus $W(S_2)$ is the classical wobbling group $W(\mathbb Z)$.

We also define $\FSym(S_n)$ to be the group of finitely supported permutations of $V(S_n)$. Finally, we define the \defin{Houghton group} $H_n$, $n\geq 2$. For $n=2$, it is the \defin{lampshuffler group} $L:=\FSym(\Z)\rtimes\langle t\rangle \leq W(\Z)$, where $t(\ell):=\ell+1$. This group has a generating set of size $2$, namely a transposition and a translation by $1$. For $n\geq 3$, the Houghton group $H_n$ is the subgroup of $\Sym(S_n)$ (and supergroup of $\FSym(S_n)$) generated by elements $g_2,\dots,g_n$ that act as follows: $g_k$ translates all the vertices in $V(R_1)\cup V(R_k)$ by $1$ towards the end of $R_1$. Formally,  \[g_k\cdot (1,m)=(1,m+1),\qquad g_k\cdot 0=(1,1),\qquad g_k\cdot (k,1)=0,\] \[g_k\cdot (k,m)=(k,m-1) \text{ (for $m\geq 2$) },\qquad g_k\cdot (i,m)=(i,m)\text{ (for $i\neq 1,k$) }.\]   

\begin{thm}\label{thm:wobbling-examples}
For every $n\geq2$, each of $W(S_n)$, $\FSym(S_n)$, and $H_n$ has finite twin-width and infinite uniform twin-width. For $n\geq3$, the group $H_n$ is finitely presented.
\end{thm}

\begin{proof}
Order $V(S_n)$ by $0\prec(1,1)\prec(2,1)\prec\cdots\prec(n,1)
  \prec(1,2)\prec\cdots$, and let $\iota\colon(V(S_n),\prec)\to(\mathbb N,<)$ be the order isomorphism. For all $a,b\in V(S_n)$, $|\iota(a)-\iota(b)|
  \leq n\,d_{S_n}(a,b)+1$. Fix $\sigma\in W(S_n)$ and put $C=\|\sigma\|_w$. The conjugate permutation $\bar\sigma:=\iota\sigma\iota^{-1}$
of $\mathbb N$ satisfies \[\|\bar\sigma\|_w
  :=\sup_{m\in\mathbb N}|\bar\sigma(m)-m|
  \leq B:=nC+1.\]
If $m_1<\cdots<m_r$ and
  $\bar\sigma(m_1)>\cdots>\bar\sigma(m_r)$,
then $m_r-B
  \leq\bar\sigma(m_r)
  <\bar\sigma(m_1)
  \leq m_1+B$.
Hence, $m_r-m_1<2B$, while $m_r-m_1\geq r-1$, and therefore $r\leq2B$. Thus, every element of $W(S_n)$ has finite decreasing width in the fixed well-order $\prec$. The action is faithful, so \cref{thm:finite-shuffle}(a) shows that $W(S_n)$ has finite twin-width. The same follows similarly for its subgroups $\FSym(S_n)$ and $H_n$, as the action of $W(S_n)$ on $\N$ induces actions of $\FSym(S_n)$ and $H_n$ on $\N$.

Since $V(S_n)$ is countably infinite, $\FSym(S_n)$ contains a copy of $\Alt(m)$ for every $m$. By \cref{cor:finite-unbounded}, their uniform twin-widths are unbounded. Subgroup monotonicity therefore gives $\utww(\FSym(S_n))
  =\utww(H_n)
  =\utww(W(S_n))
  =\infty$.
Finally, Brown proved that $H_n$ is finitely presented for $n\geq3$ \cite{brown1987finiteness}.
\end{proof}

The translation vectors in all $n$ rays define a surjective homomorphism
\[
  \tau\colon H_n\longrightarrow
  \left\{(z_1,\ldots,z_n)\in\mathbb Z^n:
  \sum_{i=1}^n z_i=0\right\}
  \cong\mathbb Z^{n-1}.
\]
Surjectivity follows from the vectors $e_1-e_k$ of the permutations $g_k$, and $\ker\tau=\FSym(S_n)$. Hence
\[
  1\longrightarrow\FSym(S_n)
  \longrightarrow H_n
  \longrightarrow\mathbb Z^{n-1}
  \longrightarrow1
\]
is exact. Since $\FSym(S_n)$ is a directed union of finite groups, it is elementary
amenable; the quotient $\mathbb Z^{n-1}$ is abelian. Hence $H_n$ is elementary
amenable, and in particular amenable.

One might note that we can already easily find a simple group with finite twin-width and infinite uniform twin-width: it is $\FAlt(\Z)$, the group of even permutations of finite support (this is a classical result, see \cite[Theorem~3.2, p.~458]{Hall2006}). However, this group is not even finitely generated. We will next obtain a finitely presented example.  

\section{Thompson's groups}\label{sec5}

Among the Houghton groups $H_n$ with $n\geq2$, the lampshuffler group $H_2$ is the unique one that is not finitely presented \cite{brown1987finiteness}. Thompson's group $V$, on the other hand, is finitely presented, infinite, and simple; see \cite[Section~6]{CFP} and \cite[Theorem~1.3]{BleakQuick}. It is also nonamenable \cite{higman1974finitely}. We prove that it has finite twin-width and infinite uniform twin-width.

Let
\[
  \{0,1\}^{<\mathbb N}:=\bigcup_{n\geq0}\{0,1\}^n,
  \qquad
  \C:=\{0,1\}^{\mathbb N}.
\]
Thus $\C$ is the Cantor space with its product topology. If $u\in\{0,1\}^{<\mathbb N}$ and $z\in\C$, write $uz$ for their concatenation and put $u\C:=\{uz:z\in\C\}$.
A finite set $U=\{u_1,\ldots,u_m\}\subseteq\{0,1\}^{<\mathbb N}$ is a \defin{complete prefix code} if $\C=\bigsqcup_{i=1}^m u_i\C$.
Equivalently, if every $x\in\C$ has a unique expression $x=u_i z$ with $i\in[m]$ and $z\in\C$.

We use the standard prefix-table model of $V$. A prefix table consists of complete prefix codes $U=\{u_1,\ldots,u_m\}$, $W=\{v_1,\ldots,v_m\}$, together with a bijection $u_i\mapsto v_i$. It defines a homeomorphism $g\colon\C\to\C$ by
$g(u_i z)=v_i z$, where $z\in\C$.
The pair $u_i\mapsto v_i$ is a \defin{branch}. 
Replacing one branch by
$u_i0\mapsto v_i0$, $u_i1\mapsto v_i$ is an \defin{expansion} and does not change the represented homeomorphism. Two tables represent the same element if, after reordering branches, they admit a common expansion.

Thompson's group $V$ is the group of homeomorphisms of $\C$
represented by prefix-replacement tables, or equivalently the set of
equivalence classes of such tables.
\begin{lem}
\label{known:V-model}
The equivalence classes of finite complete binary prefix tables under common expansion form Thompson's group $V$.
\end{lem}

\begin{proof}
This is the tree-pair description of $V$ in Cannon--Floyd--Parry
\cite[Section~6, pp.~240--248]{CFP}; a finite binary word records the path from the root to a leaf, and expansion corresponds to replacing a leaf by
its two children.  See also Bleak--Quick
\cite[Section~2]{BleakQuick}.
\end{proof}

Let $\C_{00}:=\{x\in\C:x\text{ is eventually }0\}$, i.e., $\C_{00}$ consists of those infinite binary sequences having only finitely many entries equal to $1$.  
In other words, $x\in \C_{00}$ if and only if  $x=u0^\infty$
for some finite binary word $u$, where $0^\infty:=000\cdots$. We define $c(x):=u$. For the exceptional sequence $0^\infty$, which contains no occurrence of $1$, we put $c(0^\infty):=\varepsilon$, the empty word.
We call $c(x)$ the \defin{canonical code} of $x$. We also fix on $\C_{00}$ the \defin{shortlex order}, that is, shorter words are less than longer words, whereas the lexicographic order is used to compare words of equal length. This is a well-order.

\begin{lem}\label{prop:V-shuffle}
If $g\in V$ has a prefix table with $m$ branches, then $\dw_{<}(g|_{\C_{00}})\leq m$.
Consequently, $V$ has finite twin-width.
\end{lem}

\begin{proof}
The set $\C_{00}$ is invariant under prefix replacement. Fix a branch  $u_i z\longmapsto v_i z$.
If $z\neq0^\infty$, then we have $c(u_i z)=u_i c(z)$, and  $c(v_i z)=v_i c(z)$.
Thus, for $z,z'\neq0^\infty$, shortlex comparison is preserved: on each side a fixed prefix length is added, and equal-length lexicographic comparison ignores the common prefix. Moreover, $u_i0^\infty$ is the least point of $u_i\C_{00}$, while $v_i0^\infty$ is the least point of $v_i\C_{00}$. Therefore the restriction of the branch to $u_i\C_{00}$ is increasing.

The permutation $g|_{\C_{00}}$ is consequently covered by $m$ increasing partial permutations. By \cref{lem:monotone-cover},
$\dw_{<}(g|_{\C_{00}})\leq m$.
The set $\C_{00}$ is dense in $\C$, and every element of $V$ is continuous. Hence an element fixing $\C_{00}$ pointwise is the identity, so the action on $\C_{00}$ is faithful. The conclusion follows from \cref{thm:finite-shuffle}(a).
\end{proof}

\begin{lem}\label{lem:V-symmetric}
Thompson's group $V$ has infinite uniform twin-width.
\end{lem}

\begin{proof}
The words of length $n$ form a complete prefix code. Permuting their cylinders by prefix replacements embeds $\Sym(2^n)\leq V$, and hence $\Alt(2^n)\leq V$.
By \cref{cor:finite-unbounded}, the uniform twin-widths of these alternating groups tend to infinity. Subgroup monotonicity gives $\utww(V)=\infty$.
\end{proof}

While we are at it, let us settle the question of twin-width and uniform twin-width for the other two Thompson's groups as well. We know that Thompson's group $F$ has uniform twin-width $2$, as it is right-orderable (see \cite[Proposition 5.2]{BGTT}). So, there remains only the group $T$ to examine. Since $T\leq V$ (through \cref{known:faithful-action}), $T$ also has finite twin-width.

A group is \defin{circularly orderable} if it admits a circular ordering that is preserved by right self-action (in fact, left or right does not matter; if $c_\ell$ is a left-invariant circular order on a group $G$, then $c_{\mathrm r}(x,y,z):=c_{\ell}(x^{-1},y^{-1},z^{-1})$
is a right-invariant circular order). We prove the following.

\begin{lem}\label{lem:cyclicorder}
Every nontrivial circularly orderable group $G$ satisfies
$\operatorname{utww}(G)=2$.
\end{lem}

\begin{proof}
Let $c$ be a right-invariant circular order on $G$. Cut $c$ at the
identity $e$ to obtain the total order $\prec$ defined by
\[
  e\prec x\quad(x\neq e),
  \qquad
  x\prec y:\iff c(e,x,y)
  \quad(x,y\neq e).
\]

We fix $g\neq e$. The two sets $I^-:=\{x\in G:x\prec g^{-1}\}$, and $I^+:=\{x\in G:g^{-1}\preceq x\}$ are intervals.  Since right multiplication
$R_g\colon x\mapsto xg$ preserves $c$, the sets $ R_g(I^+)$ and $R_g(I^-)$ are also intervals, and as $e\in R_g(I^+)$, we get $R_g(I^+)\prec R_g(I^-)$. Consequently, the bijection matrix of $R_g$ is obtained from the
$2\times2$ anti-diagonal permutation matrix by substituting identity matrices for the two $1$-entries. Each of these matrices has strict
twin-width at most $2$, so \cref{lem:sub} gives $\operatorname{stww}(M_{R_g})\leq 2$.
The identity translation is monotone and satisfies the same bound. Thus, the order $\prec$ witnesses
  $\operatorname{utww}(G)\leq 2$.
Since $G$ is nontrivial and $2$ is the minimum strict twin-width of a
nontrivial bijection matrix, equality follows.
\end{proof}

In its standard realization, $T$ is the countable group of
orientation-preserving dyadic piecewise-linear homeomorphisms of $S^{1}$; see \cite{CFP}. Moreover, it is known that a countable group is circularly orderable if and only if it embeds in
$\operatorname{Homeo}_{+}(S^{1})$, see \cite[Proposition~4.12]{CMR}. By \cref{lem:cyclicorder}, we have our final theorem.  
\begin{thm}
Thompson's group $T$ satisfies $\operatorname{utww}(T)=2$.
\end{thm}

%%%%%%%%%%%%%%%%%%%%%%%%%%%%%%%%%%%%%%%%%%%%%%%%%%%%%%%%%%%%%%%%%%%%%%%%%%%%%%%%%%%%%%%%%%%%%%%%%%%

\bibliographystyle{plainurlnat}
\bibliography{ref}

\newpage
    
\end{document}